\documentclass[12pt,reqno,intlim]{amsart}

\usepackage{amssymb,amsmath}

\usepackage[arrow,matrix,curve]{xy}

\newdir{ >}{{}*!/-5pt/\dir{>}}

\newcommand{\nc}{\newcommand}

\nc{\bC}{\bold{C}} \nc{\bN}{\Bbb{N}} \nc{\cF}{\mathcal{F}}
\nc{\cE}{\mathcal{E}} \nc{\cR}{\mathcal{R}} \nc{\cM}{\mathcal{M}}
\nc{\al}{\alpha} \nc{\bt}{\beta} \nc{\gm}{\gamma} \nc{\dl}{\delta}
\nc{\om}{\omega} \nc{\sg}{\sigma} \nc{\Sg}{\Sigma} \nc{\vf}{\varphi}
\nc{\ve}{\varepsilon} \nc{\os}{\overset} \nc{\ol}{\overline}
\nc{\ul}{\underline} \nc{\us}{\underset} \nc{\sbs}{\subset}
\nc{\bsl}{\backslash} \nc{\Ra}{\Rightarrow}
\nc{\lra}{\longrightarrow} \nc{\all}{\allowdisplaybreaks}

\nc{\Codes}{\operatorname{{\bold{Codes}}}}
\nc{\RegMono}{\operatorname{\mathcal{R}{\rm{eg}\mathcal{M}{\rm{ono}\!}}}}
\nc{\RegEpi}{\operatorname{\mathcal{R}{\rm{eg}\mathcal{E}{\rm{pi}\!}}}}
\nc{\Mn}{\operatorname{\mathcal{M}{\rm{ono}\!}}}
\nc{\Ep}{\operatorname{\mathcal{E}{\rm{pi}\!}}}
\nc{\Rg}{\operatorname{\mathcal{R}{\rm{eg}\!}}}
\nc{\Ob}{\operatorname{Ob\!}}

\numberwithin{equation}{section}

\newtheorem{theo}{\ \ \ Theorem}[section]
\newtheorem{lem}[theo]{\ \ \ Lemma}
\newtheorem{prop}[theo]{\ \ \ Proposition}
\newtheorem{cor}[theo]{\ \ \ Corollary}
\newtheorem{definition}[theo]{\ \ \ Definition}

\theoremstyle{definition}
\newtheorem{exmp}[theo]{\ \ \ Example}

\theoremstyle{remark}
\newtheorem{rem}[theo]{\ \ \ Remark}

\begin{document}

\title[]
{Right-cancellable protomodular algebras}

\author{Dali Zangurashvili}

\maketitle

% Abstract.
\begin{abstract}
A new protomodular analog of the classical criterion for the existence of a group term in the algebraic theory of a variety of universal algebras is given. To this end, the notion of a right-cancellable protomodular algebra is introduced. It is proved that the algebraic theory of a variety of universal algebras contains a group term if and only if it contains protomodular terms with respect to which all algebras from the variety are right-cancellable. This, in particular, gives a partial answer to the extended version of an open problem from loop theory whether any Hausdorff topological (semi-)loop is completely regular. Moreover, the right-cancellable algebras from the simplest protomodular varieties are characterized as sets with principal group actions as well as groups with simple additional structures. 

\bigskip

\noindent{\bf Key words and phrases}: protomodular variety;
right-cancellable protomodular algebra; group term.

\noindent{\bf 2020  Mathematics Subject Classification}: 18C05, 08B05, 22A30, 20N05.
\end{abstract}

% 1.
\section{Introduction}

It is well-known that the algebraic theory of a variety of universal algebras contains a group term (i.e. a binary term that satisfies the group identities  for some unary and  0-ary terms) if and only if it contains a 0-ary term and an associative Mal'cev term, i.e. a Mal'cev term that satisfies the identity:
\begin{equation}
p(x,t,p(s,y,z))=p(p(x,t,s),y,z).
\end{equation}

\noindent This is explained e.g. in \cite{C} and in \cite{JP}. 

Recall that one of important classes of Mal'cev varieties is that of protomodular varieties. The notion of a protomodular variety, and, more generally, of a protomodular category was introduced by Bourn in \cite{B} as an abstract setting where certain properties of groups remain valid, and, in particular, there is an intrinsic notion of a Bourn normal subobject. The class of protomodular categories contains (but does not coincides with) the class of semi-abelian categories where one can develop isomorphism and decomposition theorems, radical and commutator theory, homology of non-abelian structures \cite{JMT}. 

There is a syntactical characterization of protomodular varieties due to Bourn and Janelidze \cite{BJ}. It requires the existence of a term $\theta$ of an arbitrarily high arity $(n+1)$, binary terms $\alpha_i$ and 0-ary terms $e_i$ ($i=1,2,...,n)$ that satisfy certain identities (which, in fact, first appeared in Ursini's paper \cite{U1}, but for purely universal-algebraic reasons and only in the particular case were all $e_i$'s are equal to one another). For convenience, we call these terms protomodular.

In the previous paper \cite{Z} we give a protomodular analog of the above-mentioned criterion for the existence of a group term in a variety of universal algebras. It requires a certain kind of higher associativity condition on a protomodular term $\theta$ and the fulfillment of the identities $e_1=e_2=...=e_n$. In fact, in \cite{Z} we also give the second protomodular analog of that criterion, but it involves a rather cumbersome identity. %containing all terms $\theta$ and $\alpha_i$ $(i=1,2,...,n)$. 

In the present paper we give yet another criterion for the existence of a group term in a variety of universal algebras. It is also formulated in protomodular terms, but, as different from our first criterion, involves no requirement related to 0-ary terms, and, moreover, the identitites involved in it are simpler than the identity from our second criterion. 

For the proposed criterion we introduce the notion of a \textit{right} - \textit{cancellable} protomodular algebra (with respect to protomodular terms $\theta$ and $\alpha_i$). We define it as a protomodular algebra that satisfies the additional identities
\begin{equation}
\alpha_{i}(\theta(a_{1},a_{2},...,a_{n},b),\theta(a'_{1},a'_{2},...,a'_{n},b))=
\end{equation}
$$=\alpha_{i}(\theta(a_{1},a_{2},...,a_{n},b'),\theta(a'_{1},a'_{2},...,a'_{n},b')),$$\vskip+1mm
\noindent for all $i=1,2,...,n$ (and follow this definition throughout the paper, in spite of the fact that for $n=1$ it is not equivalent to the traditional one of right-cancellability ($ab=cb\Rightarrow a=c$)). 

It is proved that for any element $u$ of a right-cancellable algebra $A$ the binary operation defined by
\begin{equation}
a\cdot_{u}b=\theta(\alpha_{1}(a,u),\alpha_{2}(a,u),...,\alpha_{n}(a,u),b)
\end{equation}\vskip+2mm
\noindent is a group operation with unit $u$ and the inverses given by
\begin{equation}
a^{-_{u}1}=\theta(\alpha_{1}(u,a),\alpha_{2}(u,a),...,\alpha_{n}(u,a),u).
\end{equation}

\noindent This fact yields  the new protomodular analog of the above-mentioned group term existence criterion that reads as follows:\vskip+2mm

\textit{The algebraic theory of a  variety of universal algebras contains a group term if and only if it contains protomodular terms with respect to which all algebras from this variety are right-cancellable.}\vskip+2mm

%Here the group terms are obtained by replacing $u$ by any $e_i$ in the right-hand part of formula (1.3). 
 It should be noted that the Mal'cev operation corresponding to the group operation (1.3) coincides with that given by Bourn and Janelidze in \cite{BJ}, and by Ursini in \cite{U}. Hence the fact that (1.3) is a group operation implies in particular that this Mal'cev operation is associative on  right-cancellable protomodular algebras.

 Note also that the fact that (1.3) is a group operation relates to an open problem from loop theory whether the classical result of topological group theory by which any Hausdorff topological group is completely regular holds true for loops too (\cite{CPS}, Problem IX.1.17). Although our result does not provide a solution of this problem as such (because, as is shown, any loop which is right-cancellable in the sense of this paper is a group), nevertheless it gives a partial answer to the extended version of the problem, namely whether any Hausdorff topological left semi-loop is completely regular. It is shown that the class of left semi-loops which are not groups, but for which the answer to the latter problem is positive is quite wide: such left semi-loops can be obtained in a simple way  from arbitrary  groups $G$ equipped with any non-trivial bijections $\sigma:G\rightarrow G$ (that are not necessarily automorphisms). This is accomplished by characterizing the right-cancellable algebras from the simplest protomodular varieties (i.e. the varieties where the signature contains only protomodular operations, while the defining identities are precisely identities (2.1) and (2.2)) as groups with simple additional structures. Such algebras are characterized also as sets with principal group actions. 

% 2
\section{Preliminaries}

For the definition of a protomodular category we refer the reader to the paper \cite{B} by Bourn.

Let $\mathbb{V}$ be a variety of universal algebras of type $\cF$.

\begin{theo}(Bourn-Janelidze \cite{BJ})
A variety $\mathbb{V}$ is protomodular if and only if its algebraic theory contains, for some  natural $n$, 0-ary terms $e_{1}, e_{2},...,e_{n}$, binary terms $\alpha_{1}$,
$\alpha_{2}$,..., $\alpha_{n}$, and an $(n+1)$-ary term $\theta$ such that the following identities are satisfied:
\begin{equation}
\alpha_{i}(a,a)=e_{i};
\end{equation}
\begin{equation}
\theta(\alpha_{1}(a,b),\alpha_{2}(a,b),...,\alpha_{n}(a,b),b)=a.
\end{equation}
\end{theo}
\vskip+3mm

For simplicity, algebras from a protomodular variety are called protomodular algebras. The terms 
$\theta$, $\alpha_i$ and $e_i$ satisfying (2.1) and (2.2) are called protomodular.

The motivating example of a protomodular variety is given by the variety of groups. More generally, any variety whose algebraic theory contains a group term is protomodular (which in particular implies that the variety of Boolean algebras is protomodular). In that case we have:
\begin{equation}
\theta(a,b)=ab,
\end{equation}  
\begin{equation}
\alpha_{1}(a,b)=a/b.
\end{equation}

\noindent and $e$ is unit of the group. Note that the variety of Boolean algebras also has another system of protomodular terms \cite{BC1} which are given by: 
\begin{equation}
\theta(a,b,c)=(a\vee c)\wedge b,
\end{equation}
\begin{equation}
\alpha_1(a,b)=a\wedge \neg b,
\end{equation}
\begin{equation}
\alpha_2(a,b)=a\vee \neg b,
\end{equation}

\noindent with $e_1=0$ and
$e_2=1$.

Other examples of protomodular varieties are provided by the varieties of left/right semi-loops, loops, locally Boolean distributive lattices \cite{BC}, Heyting algebras, Heyting semi-lattices \cite{J}. Observe that the terms (2.3) and (2.4) serve as protomodular terms in the case of left semi-loops and loops, too.

Let $\mathbb{V}$ be a protomodular variety of universal algebras. Identities (2.1) and (2.2) immediately imply  \cite{BJ}, \cite{U}, \cite{BC1} the following:\vskip+2mm

(a) if $\alpha_{i}(a,c)=\alpha_{i}(b,c)$, for all $i$ ($1\leq i\leq
n$), then $a=b$;\vskip+2mm

(b) if $\alpha_{i}(a,b)=e_{i}$, for all $i$ ($1\leq i\leq n$), then
$a=b$;\vskip+2mm

(c) the identity
\begin{equation}
\theta(e_{1},e_{2},...,e_{n},a)=a;
\end{equation}
\vskip+4mm

(d) the term  
\begin{equation}
p(a,b,c)=\theta(\alpha_{1}(a,b),\alpha_{2}(a,b),..., \alpha_{n}(a,b),c). 
\end{equation}
\hspace{1cm} is a Mal'cev term.\vskip+2mm

If, in addition, $e_1=e_2=...=e_n$, then $\mathbb{V}$ is ideal determined \cite{U}. Recall that this means that any ideal of any $\mathbb{V}$-algebra is the kernel of a unique congruence (Gumm-Ursini \cite{GU}). Moreover, for any congruence $R$ of an $\mathbb{V}$-algebra $A$, the equivalence class  containing an element $a\in A$ is equal to the set $\theta(I,I,...,I,a)$ of all elements of the form $\theta(b_1,b_2,...,b_n,a)$, where $b_1,b_2,...,b_n\in I$, while $I$ is the kernel of $R$ \cite{B1}.\vskip+2mm

Let $A$ be a set. For any element $u\in A$, there is a bijection between the set of associative Mal'cev operations on $A$ and the set of group operations on $A$ with unit $u$. The bijection is given as follows: if $p$ is an associative Mal'cev operation, then the group operation is given by $ab=p(a,u,b)$ with $a^{-1}=p(u,a,u)$. If $\cdot$ is a group operation with unit $u$, then the corresponding Mal'cev operation is given by $p(a,b,c)=ab^{-1}c$ \cite{JP} (see e.g. \cite{C} and\cite{JP}).
\vskip+2mm

Let as above $A$ be a set  and let $\theta$ be an $(n+1)$-ary operation on $A$. The operation $\theta$ is called 1-associative \cite{Z} if by shifting the internal symbol $\theta$ together with the attached parentheses in the expression $$\theta(a_{1},a_{2},...,a_{n}, \theta(b_{1},b_{2},...,b_{n}, c))$$ \noindent to any position, the value of the expression does not change, for any $a_{1},a_{2},...,a_{n}$, $b_{1},b_{2},$ $...,b_{n},$ $c\in A$. 

The operation $\theta$ is called 2-associative (or consociative) \cite{Z} if the equality  
\begin{equation}\theta(a_{1},a_{2},...,a_{n}, \theta(b_{1},b_{2},...,b_{n}, c))
=\end{equation}
$$\theta(\theta(a_{1},a_{2},...,a_{n}, b_{1}),\theta(a_{1},a_{2},...,a_n, b_{2}),...,\theta(a_{1},a_{2},...,a_{n},b_{n}), c)$$
\vskip+1mm

\noindent holds for any $a_{1},a_{2},...,a_{n}, b_{1},b_{2},...,b_{n}, c\in A$. This definition is motivated by the fact that, for a binary operation $\times$ on $A$, the associativity  can be formulated as the condition that the mapping $(-\times a)$ from $A$ to the algebra $Map(A,A)$ of mappings $A\rightarrow A$ (with the composition operation) preserves the operation $\times$, for any $a\in A$. Generalizing this condition to the case of an operation $\theta$ of high arity, we obtain precisely identity (2.10).

For ternary operations, the 1-associativity implies associativity (1.1) (below we will refer to the latter condition as "associativity"). No other implications hold between the above three associativity(-like) conditions -- associativity, 1-associativity, and 2-associativity \cite{Z}.

A protomodular algebra is called 1-associative (resp. 2-associative or consociative) if its protomodular operation $\theta$ is such. 

A protomodular term $\theta$ is called 2-associative or consociative if it satisfies identity (2.10).\vskip+2mm

Let $A$ be a protomodular algebra and $b\in A$. We denote by $\theta_b$ the mapping $A^{n}\rightarrow A$ sending an $n$-tuple $(a_{1},a_{2},...,a_{n})$ to $\theta(a_{1},a_{2},...,a_{n}, b)$.

A protomodular algebra $A$ is called strict \cite{Z} if $\theta_b$ is a bijection, for any $b\in A$, or equivalently, if the following identity holds in $A$:
\begin{equation}
\alpha_i(\theta(a_1,a_2,...,a_n,b),b)=a_i,
\end{equation}

\noindent for any $i$ $(i=1,2,...,n)$.\vskip+1mm

Throughout the paper $\mathbb{V}_{n}$ denotes the simplest protomodular variety, i.e., the variety with the signature $\cF_{n}$ containing only one $(n+1)$-ary operation symbol $\theta$, the binary operation symbols $\alpha_{1}, \alpha_{2},...,\alpha_{n}$, and the 0-ary operation symbols $e_{1},e_{2},...,e_{n}$, where the defining identities are (2.1) and (2.2).

Strict $\mathbb{V}_{1}$-algebras are precisely left semi-loops \cite{Z}.

\vskip+3mm

% 2
\section{Right-cancellable protomodular algebras}

\begin{definition}
Let $\mathbb{V}$ be a protomodular variety, and $A$ be its algebra. $A$ is called right-cancellable with respect to protomodular terms $\theta$ and $\alpha_i$ (or simply right-cancellable if no confusion might arise) if the identity 

\begin{equation}
\alpha_{i}(\theta(a_{1},a_{2},...,a_{n},b),\theta(a'_{1},a'_{2},...,a'_{n},b))=
\end{equation}
$$=\alpha_{i}(\theta(a_{1},a_{2},...,a_{n},b'),\theta(a'_{1},a'_{2},...,a'_{n},b')),$$\vskip+3mm

\noindent holds in $A$ for any $i$ ($1\leq i\leq n$).
\end{definition}\vskip+4mm

\begin{exmp}
Groups considered as protomodular algebras obviously are right-cancellable. Moreover, there are left semi-loops which are not groups, but are right-cancellable. Take for instance the set $A=\lbrace 0,1,2\rbrace$ with the following binary operations $\theta$ and $\alpha$ equal to each other:
\begin{table}[ht]
%\centering
\begin{tabular}{|c|ccc|}
\hline \textbf{$\theta=\alpha$}&0&1&2\\
\hline 0&0&1&2\\
 1&2&0&1\\
 2&1&2&0\\
\hline
\end{tabular}
\end{table}

One can easily verify that $A$ is a right-cancellable left semi-loop. 

It should be however noted that not all left semi-loops are right-cancellable. For instance, the (transpose of the right) semi-loop given in the example of \cite{BC2} is not right-cancellable.
\end{exmp}\vskip+2mm

\begin{exmp}  Let $A=\{0,1\}$. Let us introduce the structure of $\mathbb{V}_{2}$-algebra on $A$ as
follows. Let $\theta(i,j,k)=k$ if $i\neq j$ and $\theta(i,j,k)=1-k$
if $i=j$. Moreover, let $\alpha_{1}(i,j)$ be $0$, for any $i,j$; let
$\alpha_{2}(i,j)$ be 0, if $i\neq j$, and be 1, if $i=j$. Besides,
let $e_{1}=0$ and $e_{2}=1$.

The condition (3.1) is obviously satisfied for $i=1$. It also is
satisfied for $i=2$ since the value of the left-hand side of (3.1)
depends only on whether $a_{1}=a_{2}$ and $a'_{1}=a'_{2}$. This
implies that $A$ is a right-cancellable $\mathbb{V}_{2}$-algebra.
\end{exmp}\vskip
+1mm

\begin{exmp}
Let $A=\lbrace 0,1,2,3\rbrace$, and $n=3$. We define operations on the set $A$ by means of the tables below, where $t_0$ is a triple $(a_1,a_2,a_3)$ from $A^3$ with $a_1,a_2,a_3\leq 1$, $t_1$ is a triple with $a_1\leq 1$ and either $a_2>1$ or $a_3>1$, $t_2$ is a triple with $a_1>1$ and either $a_2\leq 1$ or  $a_2\leq 1$, while $t_3$ is a triple with $a_1,a_2, a_3> 1$.  
\begin{table}[ht]
%\centering
\begin{tabular}{|c|cccc|}
\hline \textbf{$\theta$}&0&1&2&3\\
\hline $t_0$&0&1&2&3\\
 $t_1$&1&0&3&2\\
 $t_2$&2&3&0&1\\
 $t_3$&3&2&1&0\\
\hline
\end{tabular}
\end{table}

\begin{table}[ht]
%\centering
\begin{tabular}{|c|cccc|}
\hline \textbf{$\alpha_1$}&0&1&2&3\\
\hline 0&0&1&3&2\\
1&1&0&2&3\\
 2&3&2&0&1\\
 3&2&3&1&0\\
\hline
\end{tabular}
%\end{table}
\quad
%\begin{table}[ht]
%\centering
\begin{tabular}{|c|cccc|}
\hline \textbf{$\alpha_2$}&0&1&2&3\\
\hline 0&0&2&0&3\\
1&2&0&3&0\\
 2&0&3&0&2\\
 3&3&0&2&0\\
\hline
\end{tabular}
%\end{table}
\quad
%\begin{table}[ht]
%\centering
\begin{tabular}{|c|cccc|}
\hline \textbf{$\alpha_3$}&0&1&2&3\\
\hline 0&1&1&2&2\\
1&1&1&2&2\\
 2&2&2&1&1\\
 3&2&2&1&1\\
\hline
\end{tabular}
\end{table}

Let $e_1=e_2=0$ and $e_3=1$.
It can be verified that the set $A$ equipped with these operations is a right-cancellable $\mathbb{V}_3$-algebra.
\vskip+2mm
\end{exmp}
\vskip+4mm

At the end of Section 5 we will show how to construct right-cancellable $\mathbb{V}_n$-algebras of an arbitrary cardinality, for any $n$.
\vskip+3mm

 Below, when dealing with the case $n=1$, we will use the traditional abbreviation $ab$ for $\theta(a,b)$, and 
$a/b$ for $\alpha(a,b)$.
\begin{lem}
Let $\mathbb{V}$ be a protomodular variety with $n=1$, and let $A$ be an algebra from this variety with unit $e$. The  following conditions are equivalent:

(i) $A$ is right-cancellable and,  for any $a\in A$, we have:
\begin{equation}
ae=a;
\end{equation}

(ii)  $A$ is right-cancellable and,  for any $a\in A$, we have:
\begin{equation}
a/e=a;
\end{equation}

(iii) $A$ is a group.

\end{lem}

\begin{proof}
 (i)$\Rightarrow$(ii): From (2.2) we have $(a/e)e=a$; on the other hand, (3.5) implies $(a/e)e=a/e$.
 
(i)$\Rightarrow$(iii):  From (3.5) we obtain the identity
\begin{equation}
(ab)/(a'b)=a/a'.
\end{equation}

\noindent We already know that identity (3.5) implies (3.6). Hence from (3.7) we obtain

$$(ab)c/(bc)=(ab)/b=ab/eb=a/e=a,$$

\noindent for any $a,b,c\in A$. Multiplying both parts of this equality by $(bc)$
from the right, we obtain the associativity condition. Then taking into account the fact that any set equipped with an associative binary operation that has a left identity and left inverses, is a group (see, e.g. \cite{F}), we obtain that $A$ is a group. 

The implications (ii)$\Rightarrow$(i) and (iii)$\Rightarrow$(i) are obvious.
\end{proof}

In particular, we obtain

\begin{prop}
Any loop which is right-cancellable in the sense of Definition 3.1 is a group.
\end{prop}\vskip+2mm

Proposition 5.7 that will be given at the end of the paper implies

\begin{prop}
Let $\mathbb{V}$ be a protomodular variety with $n=1$. Any finite $\mathbb{V}$-algebra which is right-cancellable in the sense of Definition 3.1 is a left semi-loop.
\end{prop}

\begin{rem} An algebra can be right-cancellable with respect to one system of protomodular terms, but might be not such with respect to some other system of protomodular terms. The example is provided by non-trivial Boolean algebras. They obviously are right-cancellable with respect to the protomodular terms determined by a group term of the variety of Boolean algebras. However, these algebras are not right-cancellable with respect to the protomodular terms (2.5)-(2.7), as follows from the equality $\alpha_2(\theta(0,0,b),b)=\neg b$, and identity (2.8). 
\end{rem}

\begin{rem} It can be shown that non-trivial locally Boolean distributive algebras are not right-cancellable with respect to the protomodular terms given in \cite{BC1}. Similarly, non-trivial Heyting algebras and non-trivial Heyting semi-lattices are not right-cancellable with respect to the protomodular terms given in \cite{J} .
\end{rem}

Identities (3.1) can be simplified in some cases. More precisely, we have
\begin{lem}
Let $A$ be a protomodular algebra. For the conditions below we have the implications (i)$\Rightarrow$(ii) and (i)$\Rightarrow$(iii). If $A$ is either 1-associative or 2-associative, then the conditions (i) and (ii) are equivalent:

(i) $A$ is right-cancellable;

(ii) the identity \begin{equation}
\alpha_{i}(b,\theta(a_{1},a_{2},...,a_{n},b))=
\end{equation}
$$=\alpha_{i}(b',\theta(a_{1},a_{2},...,a_{n},b')),$$\vskip+3mm

\noindent holds in $A$ for any $i$ ($1\leq i\leq n$);

(iii) the identity \begin{equation}
\alpha_{i}(\theta(a_{1},a_{2},...,a_{n},b),b)=
\end{equation}
$$=\alpha_{i}(\theta(a_{1},a_{2},...,a_{n},b'),b'),$$\vskip+3mm

\noindent holds in $A$ for any $i$ ($1\leq i\leq n$).
\end{lem}

\begin{proof}
The implications (i)$\Rightarrow$(ii) and (i)$\Rightarrow$(iii) immediately follow from identity (2.8). 

(ii)$\Rightarrow$(i): From
(2.2) we have
$$\alpha_{i}(\theta(a_{1},a_{2},...,a_{n},b),\theta(a'_{1},a'_{2},...,a'_{n},b))=$$
$$\alpha_{i}(\theta(a_{1},a_{2},...,a_{n},b), \theta(a'_{1},a'_{2},...,a'_{n},\theta(c_{1},c_{2},...,c_{n},\theta(a_{1},a_{2},...,a_{n},b)))),$$\vskip+2mm 
\noindent where
$$c_{i}=\alpha_{i}(b,\theta(a_{1},a_{2},...,a_{n},b)).$$

\noindent By (3.5) the value of $c_{i}$ does not depend on $b$. Applying the 1-associativity (resp. the 2-associativity) and then again (3.5), we obtain (3.1).

\end{proof}

\begin{rem}
The implication (ii)$\Rightarrow$(i) does not hold in general. The counterexample is provided by any Moufang loop which is not a group, as follows from Proposition 3.6. 
%(we have $b/(ab)=a^{-1}$, for such loops). 

As to the implication (iii)$\Rightarrow$(i), it is not valid either, as the same counterexample shows. Moreover, we do not impose the additional condition of 1-associativity or 2-associativity on the algebras satisfying (3.6) for the following reason. If $n=1$, then, as is observed in Example 3.15 of \cite{Z}, the associativity condition on an algebra of this kind  implies that it is a group. If $n\geq 2$, then either of the 1-associativity and 2-associativity conditions implies that such an algebra is trivial, as is shown in Remark 3.16 of \cite{Z}.
\end{rem}

\section{Right-cancellable protomodular algebras as sets with principal group actions}

Let $A$ be an algebra from a protomodular variety, and $b\in A$. Recall that the symbol $\theta_b$ denotes the mapping $A^{n}\rightarrow A$ sending an $n$-tuple $(a_{1},a_{2},...,a_{n})$ to $\theta(a_{1},a_{2},...,a_{n}, b)$.
\begin{lem}
Let $A$ be a right-cancellable protomodular algebra. For any $b,b'\in A$, the mappings $\theta_b$ and $\theta_{b'}$ determine one and the same equivalence relation $\sim$ on $A^{n}$. 
\end{lem}

\begin{proof}
Assume 
\begin{equation}
\theta(a_1, a_2,..., a_n,b)=\theta(c_1, c_2,..., c_n,b).
\end{equation}
\noindent for some $a_1, a_2,..., a_n,c_1, c_2,..., c_n\in A$. Then, for any $i$, we have 
\begin{equation}
\alpha_i(\theta(a_1, a_2,..., a_n,b),\theta(c_1, c_2,..., c_n,b))=e_i.
\end{equation}

\noindent But since the algebra $A$ is right-cancellable, one can replace both $b$'s in the left-hand side of (4.2) by $b'$, and hence we obtain 
\begin{equation}
\alpha_i(\theta(a_1, a_2,..., a_n,b'),\theta(c_1, c_2,..., c_n,b'))=e_i.
\end{equation}

\noindent Since (4.3) holds for any $i$, from (b) of Section 2 we conclude that 

\begin{equation}
\theta(a_1, a_2,..., a_n,b')=\theta(c_1, c_2,..., c_n,b').
\end{equation}

\end{proof}\vskip+3mm

Let $A$ be a protomodular algebra, and $a_1, a_2,..., a_n\in A$. We denote by $\theta^{(a_1, a_2,..., a_n)}$ the mapping $A\rightarrow A$ sending $b\in A$ to $\theta(a_1, a_2,..., a_n,b)$, and use the term "translation" for a mapping of this kind.

Lemma 4.1 implies

\begin{lem}
Let $A$ be a right-cancellable algebra. If the actions of two translations coincide on at least one element of $A$, then these translations are equal.
\end{lem}

\begin{lem}
For the conditions below, the implications $(i)\Rightarrow (ii)\Rightarrow(iii)$ hold. If $A$ is strict, then the conditions (i)-(ii) are equivalent, and  the $n$-tuple $(c_1,c_2,...,c_n)$ satisfying (4.5) is unique.\vskip+2mm

(i) A is right-cancellable;

(ii) for any $a_1, a_2,..., a_n, b_1, b_2,..., b_n\in A$ there are $c_1, c_2,..., c_n\in A$ such that 

\begin{equation}
\theta^{(c_1, c_2,..., c_n)}\theta^{(b_1, b_2,..., b_n)}=\theta^{(a_1, a_2,..., a_n)}.
\end{equation}

(iii) all translations are bijections, and their inverses are also translations. 
\end{lem}

\begin{proof}
$(i)\Rightarrow (ii)$: Let $x\in A$ and

\begin{equation}
c_{i}=\alpha_i(\theta(a_1, a_2,..., a_n,x), \theta(b_1, b_2,..., b_n,x)),
\end{equation}

\noindent for any $i$. Then
\begin{equation}
\theta(c_1, c_2,..., c_n, \theta(b_1, b_2,..., b_n,x))=\theta(a_1, a_2,..., a_n,x).
\end{equation}

\noindent Since $A$ is right-cancellable, the value of $c_i$ does not depend on $x$. This implies (4.5).

$(ii)\Rightarrow (iii)$: Taking $a_1=e_1$, $a_2=e_2$,..., $a_n=e_n$, and applying (2.8) we obtain 

\begin{equation}
\theta^{(c_1, c_2,...,c_n)}\theta^{(b_1, b_2,...,b_n)}=1_A,
\end{equation}

\noindent for some $c_1, c_2,...,c_n\in A$. This implies that $\theta^{(b_1, b_2,..., b_n)}$ is injective. Since $b_1, b_2,..., b_n$ are arbitrary, $\theta^{(c_1, c_2,...,c_n)}$ is injective too. Then (4.8) implies that $\theta^{(c_1, c_2,...,c_n)}$ is bijective. Hence 
$\theta^{(b_1, b_2,...,b_n)}$ also is bijective and $\theta^{(c_1,c_2,...,c_n)}$ is its inverse.

$(ii)\Rightarrow (i)$: From the equality (4.5) it follows that (4.7) is satisfied for any $x\in A$. But, since $A$ is strict, (4.7) implies (4.6). 
\end{proof}

\begin{rem}
Note that the condition $(iii)$ of Lemma 4.3 does not imply $(i)$, even if $A$ is strict. The counterexample is provided by any loop which is not a group, as follows from Proposition 3.6. 
\end{rem}

Recall that an action of a group $G$ on a set $X$ is called transitive if, for any $x,y\in X$ there is an element $g\in G$ with $gx=y$. If, moreover, such $g$'s are unique, or, equivalently, the stabilizer of any element of $X$ is trivial, then we say that the action is $\textit{sharply transitive}$ or $principal$. In that case $G$ obviously is bijective to $X$.\vskip+2mm

Let $A$ be a right-cancellable protomodular algebra, and let $S(A)$ denote the set of all bijections (not necessarily isomorphisms) $A\rightarrow A$. Consider the mapping
$$\overline{\theta}:A^{n}\rightarrow S(A)$$

\noindent which sends an $n$-tuple $(a_1, a_2,..., a_n)$ to the bijection  $\theta^{(a_1, a_2,..., a_n)}:A\rightarrow A$.

\begin{lem}
Let $A$ be a right-cancellable protomodular algebra. Then the set  $\overline{\theta}(A^{n})$ is a subgroup of $S(A)$, and its action on $A$ is principal. Moreover, for any $a_1, a_2,..., a_n$, $b_1,b_2,...,b_n\in A$ and $u\in A$, we have: \vskip+2mm

(a)\begin{equation}
\theta^{(a_1,a_2,...,a_n)}\theta^{(b_1,b_2,...,b_n)}=\theta^{(c_1,c_2,...,c_n)},
\end{equation}
\noindent where

\begin{equation}
c_i=
\end{equation}
$$\alpha_{i}(\theta(a_1, a_2,..., a_n,u),
\theta (\alpha_{1}(u,\theta(b_1,b_2,...,b_n,u)),\alpha_{2}(u,\theta(b_1,b_2,...,b_n,u)),$$
$$...,\alpha_{n}(u,\theta(b_1,b_2,...,b_n,u)),u),$$\vskip+2mm

(b) 
\begin{equation}
{(\theta^{(a_1, a_2,..., a_n)})}^{-1}=\theta^{(b_1,b_2,...,b_n)},
\end{equation}

\noindent where
\begin{equation}
b_i=\alpha_i(u, \theta(a_1, a_2,..., a_n,u)).
\end{equation} 
\end{lem}

\begin{proof}
Identity (2.8) implies that $\overline{\theta}(A^{n})$ contains $1_A$. Lemma 4.3 implies that $\overline{\theta}(A^{n})$ is closed under inverses.

Let us show that $\overline{\theta}(A^{n})$ is closed under composition. According to Lemma 4.3, if $f,g\in \overline{\theta}(A^{n})$, then $gf^{-1}\in \overline{\theta}(A^{n})$. Moreover, $f^{-1}\in \overline{\theta}(A^{n})$. Hence $g(f^{-1})^{-1}=gf\in \overline{\theta}(A^{n})$.

Lemma 4.2 and identity (2.2) imply that the action of $\overline{\theta}(A^{n})$ on $A$ is principal.

Equalities (4.11) and (4.12) immediately follow from (4.6) (substitute $a_{i}=e_{i}$), while the equalities (4.9) and (4.10) follow from (4.6), (4.11) and (4.12).
\end{proof}

\begin{exmp}
Let $A$ be the protomodular algebra from Example 3.4. One can easily observe that there are precisely four different translations. For instance, the translations $\theta^{(1,2,1)}$ and $\theta^{(3,0,2)}$ are different and non-trivial. Moreover, either of them has order $2$. Hence the group $\overline{\theta}(A^{n})$ is isomorphic to the Klein Four Group, i.e. the product $\textbf{Z}_2\times \textbf{Z}_2$ of the two cyclic groups of order $2$.
\end{exmp}
\vskip+2mm

Let $\mathbb{X}_n$ denote the category with objects being quadruples 
\begin{equation}
(A,G,\varepsilon,\mu)
\end{equation}
\noindent where $A$ is a set, $G$ is a group acting on $A$ such that the action is principal, while $\varepsilon$ and $\mu$ are mappings resp. $A^{n}\rightarrow G$ and $G\rightarrow A^{n}$ with 
\begin{equation}
\varepsilon\mu=1_G. 
\end{equation}

\noindent Morphisms $(A,G,\varepsilon,\mu)\rightarrow (A',G',\varepsilon',\mu')$ are pairs $(\varphi,\psi)$, where $\varphi$ is a mapping $A\rightarrow A'$, while $\psi$ is a group homomorphism $G\rightarrow G'$ such that, for any $g\in G$ and $a\in A$, one has the equality
\begin{equation}
\varphi(ga)=\psi(g)\varphi(a),
\end{equation} 
\noindent and the diagram 
\begin{equation}
\xymatrix{G\ar[r]^{\mu}\ar[d]_{\psi}&A^{n}\ar[r]^{{\varepsilon}}\ar[d]^{\varphi^{n}}&G\ar[d]^{\psi}\\
G'\ar[r]_{\mu'}&A'^{n}\ar[r]_{\varepsilon'}&G'}
\end{equation}

\noindent is commutative.

Let $RC$-$\mathbb{V}_n$ be the category of right-cancellable $\mathbb{V}_n$-algebras.

\begin{prop}
The categories $RC$-$\mathbb{V}_n$ and $\mathbb{X}_n$ are equivalent. Under this equivalence the strict $RC$-$\mathbb{V}_n$-algebras correspond precisely to the quadruples $(A,G,\varepsilon,\mu)$ with bijective (and mutually inverse) $\varepsilon$ and $\mu$. If $n=1$, the groups correspond precisely to the quadruples $(A,G,\varepsilon,\mu)$ with
\begin{equation}
\varepsilon(a)\mu(e)=a,
\end{equation}
\noindent for any $a\in A$.
\end{prop}

\begin{proof}
Let us construct the functor $F:RC$-$\mathbb{V}_n\rightarrow \mathbb{X}_n$ as follows. For a right-cancellable protomodular algebra $A$, take the group $G=\overline{\theta}(A^{n})$, and define 
\begin{equation}
\varepsilon(a_1,a_2,...,a_n)=\theta^{(a_1,a_2,...,a_n)},
\end{equation}
\begin{equation}
\mu(\theta^{(a_1,a_2,...,a_n)})=(\alpha_1(\theta(a_1,a_2,...a_n,u),u),...,\alpha_n(\theta(a_1,a_2,...a_n,u),u)),
\end{equation}

\noindent where $u$ is an arbitrary element of $A$. Note that the value of each $\alpha_i(\theta(a_1,a_2,...a_n,u),u)$ does not depend on $u$, since $A$ is right-cancellable. Equality (2.2) implies that
\begin{equation}
\varepsilon\mu(\theta^{(a_1,a_2,...,a_n)})(u)=\theta^{(a_1,a_2,...,a_n)}(u). 
\end{equation}

\noindent Then from Lemma 4.2 we obtain 
\begin{equation}
\varepsilon\mu(\theta^{(a_1,a_2,...,a_n)})=\theta^{(a_1,a_2,...,a_n)}. 
\end{equation}

\noindent Hence $(A,\overline{\theta}(A^{n}), \varepsilon, \mu)$ is an object of $\mathbb{X}_n$.

 Assume $\varphi:A\rightarrow A'$ is a homomorphism of $\mathbb{V}_n$-algebras. Consider the mapping $\psi:\overline{\theta}(A^{n})\rightarrow \overline{\theta}(A'^{n})$ sending the bijection $\theta^{(a_1,a_2,...,a_n)}$ to the one $\theta^{(\varphi(a_1),\varphi(a_2),...,\varphi(a_n))}$. Lemma 4.2 implies that it is defined correctly. Moreover, equalities (4.9)-(4.10) and Lemma 4.2 imply that the pair $(\varphi,\psi)$ is a morphism of $\mathbb{X}_{n}$.
 
 Let us now construct the functor $G:\mathbb{X}_n\rightarrow RC$-$\mathbb{V}_n$. To this end consider a quadruple (4.13) from $\mathbb{X}_n$. Define the following operations on $A$:
 \begin{equation}
 \theta(a_1,a_2,...,a_n,b)=\varepsilon(a_1,a_2,...,a_n)b;
 \end{equation}
 \begin{equation}
 e_i=\pi_i\mu(e);
 \end{equation}
 \noindent and 
 \begin{equation}
 \alpha_i(a,b)=\pi_i\mu(g),
 \end{equation}
 
 \noindent where $\pi_{i}$ is the canonical projection $A^{n}\rightarrow A$, while $g$ is the unique element of $G$ with $g(b)=a$. Equality (4.14) implies that operations (4.22)-(4.24) give the structure of a $\mathbb{V}_n$-algebra on $A$. To show that this algebra is right-cancellable, we observe that 
 
 \begin{equation}
 \alpha_i(\theta(a_1,a_2,...,a_n,b),\theta(a'_1,a'_2,...,a'_n,b)=\pi_i\mu(g),
 \end{equation}
 
 \noindent where $g$ is the unique element of $G$ with
 \begin{equation}
 g(\theta(a'_1,a'_2,...,a'_n,b)=\theta(a_1,a_2,...,a_n,b).
 \end{equation}
 
\noindent Applying (4.22) we obtain
\begin{equation}
(g\varepsilon(a'_1,a'_2,...,a'_n))b=\varepsilon(a_{1},a_{2},...,a_{n})b.
\end{equation}

\noindent Since the action of $G$ is principal, (4.27) implies that
\begin{equation}
g\varepsilon(a'_1,a'_2,...,a'_n)=\varepsilon(a_{1},a_{2},...,a_{n}).
\end{equation}

\noindent Thus the element $g$ does not depend on $b$, and hence equality (3.1) is satisfied.

If a pair $(\varphi,\psi)$ is a morphism $(A,G,\varepsilon,\mu)\rightarrow (A',G',\varepsilon',\mu')$, then, as one easily can verify, $\varphi$ is a homomorphism of the corresponding $\mathbb{V}_n$-algebras $A\rightarrow A'$.

One can verify that $GF=1_{(RC-\mathbb{V}_n)_*}$. Moreover, we have a natural isomorphism $\tau:1_{\mathbb{X}_n}\rightarrow FG$. Namely, $\tau_{(A,G,\varepsilon,\mu)}$ is the pair $(1_A,\kappa)$, where $\kappa$ is the homomorphism $G\rightarrow \overline{\theta}(A^{n})$ sending $g\in A$ to the mapping $\theta^{\mu(g)}:A\rightarrow A$ (which sends $a\in A$ to $ga$).

If $G((A,G,\varepsilon,\mu))$ is a group, then (4.17) obviously is satisfied. The converse statement follows from Lemma 3.4.
\end{proof}

\section{Group operations on right-cancellable protomodular algebras}

Let $A$ be a right-cancellable protomodular algebra. As we mentioned above, each $u\in A$ gives the bijection between $A$ and $\overline{\theta}(A^{n})$: it sends an element $a$ of $A$ to the unique $g\in \overline{\theta}(A^{n})$ with $g(u)=a$. In fact, this bijection is the composition of the following two bijections 
\begin{equation}
\xymatrix{A\ar[r]^{\gamma_u} & A^{n}/\sim\ar[r] & \overline{\theta}(A^{n})},
\end{equation}

\noindent where $\sim$ is the equivalence relation from Lemma 4.1, and $\gamma_u$ sends an element $a$ of $A$ to the equivalence class of the $n$-tuple 
\begin{equation}
(\alpha_{1}(a,u),\alpha_{2}(a,u),...,\alpha_{n}(a,u)).
\end{equation}

 Transferring the group structure of $\overline{\theta}(A^{n})$ via these bijections, we obtain the group structure on $A$. We have 
\begin{prop}
Let $A$ be a right-cancellable protomodular algebra, and let $u$ be its arbitrary element. The binary operation on $A$ defined  by:
\begin{equation}
a\cdot_{u}b=\theta(\alpha_{1}(a,u),\alpha_{2}(a,u),...,\alpha_{n}(a,u),b).
\end{equation}

\noindent is a group operation. Its unit is $u$, while the inverse of an element $a\in A$ under this operation is given by:
\begin{equation}
a^{-_{u}1}=\theta(\alpha_{1}(u,a),\alpha_{2}(u,a),...,\alpha_{n}(u,a),u).
\end{equation}
\end{prop}

\begin{proof}
Substituting $a_i=\alpha_{i}(a,u)$ and $b_i=\alpha_i(b,u)$ in (4.10), we obtain the following value of $c_i$:
\begin{equation}
c_i=
\end{equation}
$$\alpha_{i}(\theta(\alpha_{1}(a,u),\alpha_{2}(a,u),...,\alpha_{n}(a,u),u),\theta(\alpha_{1}(u,b),\alpha_{2}(u,b),...,\alpha_{n}(u,b),u),$$

\noindent for any $i$. Since $A$ is right-cancellable, we can replace two $u$'s by $b$'s in the appropriate positions of the right-hand side of (5.5). Then, from (2.2) we obtain
\begin{equation}
c_i=\alpha_{i}(\theta(\alpha_{1}(a,u),\alpha_{2}(a,u),...,\alpha_{n}(a,u),b),u)
\end{equation}
\noindent Applying the equality
\begin{equation}
a\cdot_u b=
\theta_u(c_1,c_2,...,c_n),
\end{equation}

\noindent we obtain (5.3).

Equality (5.4) can be obtained from (4.9)-(4.10) by similar arguments.
\end{proof}

We can give the direct proof of the fact that the binary operation given by (5.3) is indeed a group operation on $A$. To this end let us first give
\begin{lem}
Let $A$ be a right-cancellable protomodular algebra, and $a,b,c,u$ be its arbitrary elements. For the binary operation $\cdot_{u}$ on $A$ given by (5.3) and any $i$ $(1\leq i\leq n)$, we have
\begin{equation}
\alpha_i(a\cdot_u c, b\cdot_u c)=\alpha_i(a, b).
\end{equation} 

\end{lem}

\begin{proof}
 We have 
\begin{equation}
\alpha_i(a\cdot_u c, b\cdot_u c)=
\end{equation} 
$$\alpha_i(\theta(\alpha_1(a,u),\alpha_2(a,u),..., \alpha_n(a,u),c),\theta(\alpha_1(b,u),\alpha_2(b,u),..., \alpha_n(b,u),c).$$\vskip+2mm

\noindent Replacing both $c$'s on the right-hand side of this equality by $u$, and then applying equality (2.2), we obtain $\alpha_i(a, b)$.
\end{proof}

Now we are ready to give the direct proof of Proposition 5.1.
\begin{proof} Let us introduce the new binary operation on $A$:
\begin{equation}
a/_{u}b=\theta(\alpha_{1}(a,b),\alpha_{2}(a,b),...,\alpha_{n}(a,b),u).
\end{equation}

\noindent Applying the fact that  $A$ is right-cancellable, identities (2.8) and (2.2), we obtain
\begin{equation}
\alpha_{i}((a/_{u}b), u)=\alpha_i(a,b).
\end{equation}

\noindent Then
\begin{equation}
(a/_{u}b)\cdot_u b=\theta(\alpha_{1}(a,b),\alpha_{2}(a,b),...,\alpha_{n}(a,b),b)=a.
\end{equation}

\noindent Hence the set $A$ equipped with the binary operations $\cdot_u$, $/_{u}$ and the constant $u$ is a $\mathbb{V}_1$-algebra. Moreover, Lemma 5.2 implies that this algebra is right-cancellable. But, as one can easily observe, $a\cdot_{u} u=a$, for any $a\in A$. Therefore, Lemma 3.5 implies that (5.3) is a group operation. Equality (5.4) immediately follows from (5.10).
\end{proof}

\begin{exmp}
For right-cancellable protomodular algebras with $n=1$ (in particular, for left semi-loops) the group operation (5.3) takes the form
\begin{equation}
a\cdot_u b=(a/u)b,
\end{equation}

\noindent while the inverse $a^{-_{u}1}$ takes the form $(u/a)u$.
\end{exmp}\vskip+4mm

Note that the Mal'cev operation corresponding to the group operation (5.3) is precisely the Mal'cev operation (2.5). Hence we obtain 
\begin{cor}
The Mal'cev term (2.5) is associative on right-cancellable protomodular algebras.
\end{cor}

We can summarize the protomodular analogs of the classical group term existence criterion (see Section 1). 
\begin{theo} 
Let $\mathbb{V}$ be a variety of universal algebras. The following conditions are equivalent:

(i) the algebraic theory of $\mathbb{V}$ contains a group term;

(ii) the algebraic theory of $\mathbb{V}$ contains protomodular terms with respect to which all $\mathbb{V}$-algebras are right-cancellable;

(iii) the algebraic theory of $\mathbb{V}$ contains protomodular terms $\theta$, $\alpha_i$, and $e_i$  ($i=1,2,...,n)$ such that $\theta$ is consociative, and, moreover, $e_1=e_2=...=e_n$;

(iv) the algebraic theory of $\mathbb{V}$ contains protomodular terms $\theta$ and $\alpha_i$ ($i=1,2,...,n)$ which satisfy the identity:
$$\theta(\alpha_1(a,b),...,\alpha_n(a,b),\theta(\alpha_1(c,d),...,\alpha_n(c,d),s))$$
$$=\theta(\alpha_1(\theta(\alpha_1(a,b),\ldots, \alpha_n(a,b),c),d),$$
$$\alpha_2(\theta(\alpha_1(a,b),\ldots, \alpha_n(a,b),c),d),\ldots,$$
$$\alpha_n(\theta(\alpha_1(a,b),\ldots, \alpha_n(a,b),c),d),s).$$
\end{theo}
\vskip+3mm

\begin{proof}
The equivalence of (i) and (ii) immediately follows from Proposition 5.1. The equivalence of (i), (iii) and (iv) is given in \cite{Z}.
\end{proof}\vskip+3mm

\begin{rem}
%Although, as follows from Theorem 5.5, the notions of a right-cancellable protomodular variety and of a consociative protomodular variety are equivalent, nevertheless 
In view of Theorem 5.5 there naturally arises the question whether the notions of a right-cancellable protomodular algebra and of a consociative protomodular algebra are  equivalent. The answer is negative. For instance, non-trivial Boolean algebras are consociative with respect to the protomodular terms given in \cite{BC} (see Example 3.5 of \cite{Z}), but they are not right-cancellable with respect to these terms (see Remark 3.8). Moreover, there are right-cancellable left semi-loops which are not consociative, as follows from Example 3.2.
\end{rem}

Finally we will show that right-cancellable $\mathbb{V}_n$-algebras can be characterized as groups with simple additional structures. 
Consider the category $\mathbb{Y}_{n}$ with objects being triples $(A, \sigma,\rho)$,  where $A$ is a group, $\sigma$ and $\varrho$ are (not necessarily homomorphisms, but) mappings resp. $A^{n}\rightarrow A$ and $A\rightarrow A^{n}$ such that $\sigma\varrho=1_A$.

\begin{prop}
The category $(RC$-$\mathbb{V}_n)_{*}$ of right-cancellable $\mathbb{V}_n$-algebras with fixed elements is isomorphic to the category $\mathbb{Y}_n$. Under this isomorphism, the strict right-cancellable $\mathbb{V}_n$-algebras (with fixed elements) correspond precisely to the triples  $(A, \sigma,\rho)$ where $\sigma$ and $\varrho$ are bijections. If $n=1$, then the groups (with fixed elements) correspond precisely to the triples  $(A, \sigma,\rho)$ with $\sigma=\varrho=1_A$.
\end{prop}

\begin{proof}
Let us construct the functor $P:(RC$-$\mathbb{V}_n)_{*}\rightarrow \mathbb{Y}_n$. If $A$ is a right-cancellable $\mathbb{V}_n$-algebra and $u\in A$, then define $P(A)$ as the triple $(A, \theta_u, (\alpha_1(-,u),\alpha_2(-,u),...,\alpha_n(-,u)))$, where $A$ is equipped with the group operation $\cdot_u$.

To construct the functor $Q:\mathbb{Y}_n\rightarrow$ $(RC$-$\mathbb{V}_n)_{*}$, consider a triple $(A, \sigma, \varrho)$ from $\mathbb{Y}_{n}$. Define the following operations on $A$:
\begin{equation}
\theta(a_1,a_2,...,a_n,b)=\sigma(a_1,a_2,...,a_n)b,
\end{equation}
\begin{equation}
\alpha_{i}(a,b)=\pi_i\varrho(ab^{-1}),
\end{equation}
\begin{equation}
e_i=\pi_i \varrho(e).
\end{equation}

\noindent One can easily show that the set $A$ equipped with these operations and with the fixed element $e$ is an object of the category $(RC$-$\mathbb{V}_n)_{*}$. 

To verify that the functors $P$ and $Q$  are mutually inverse, let us observe that for an algebra $A$ from $RC$-$\mathbb{V}_n$, elements $a, b$ and $u$ from $A$ and the operation $\alpha'_{i}$ of $QP(A)$, we have 
\begin{equation}
\alpha'_{i}(a,b)=\rho_{i}(a\cdot_{u}b^{-1})=\alpha_{i}((a\cdot_{u}b^{-1}),u).
\end{equation}

\noindent Proposition 5.1 and Lemma 5.2 imply that the right-hand part of (5.17) is precisely  $\alpha_{i}(a,b)$. 
%The rest of arguments applies Lemma 3.7.

If the algebra corresponding to an object $(A, \sigma,\rho)$ from $\mathbb{Y}_{n}$ is a group, then, for any $a\in A$, we have $a=\alpha(a,e)= \rho(a)$. Hence $\rho=\varepsilon=1_A$.

\end{proof}

Proposition 5.7 in particular implies Proposition 3.7. In fact, we can make the latter statement stronger.

\begin{prop}
Let $A$ be a finite algebra from a variety $\mathbb{V}$ that has protomodular terms with $n=1$. If $A$ is not a left semi-loop with respect to these protomodular terms, then the algebraic theory of $\mathbb{V}$ does not contain a term whose corresponding operation on $A$ is a group operation.
\end{prop}

\begin{proof}
For convenience, we use the abbreviation $ab$ for $\theta(a,b)$, and 
$a/b$ for $\alpha(a,b)$. There are elements $a,b\in A$ with
\begin{equation}
(ab)/ b\neq a.
\end{equation}
Let $I$ be the ideal generated by $a$ and $(ab)/ b$. Consider the mapping $\theta_b: A\rightarrow A$ sending $a$ to $ab$. The restriction of $\theta_b$ on $I$ is not injective, as follows from (5.18) and the equality
$$((ab)/b)b=ab.$$
\noindent Hence, since $I$ is finite, the cardinality of $\theta_b(I)$ is less than the cardinality of $I$. But $\theta_b(I)=Ib$ is an equivalence class of the congruence induced by $I$ (see Section 2). This completes the proof.
\end{proof}

\section{Acknowledgment}
Financial support from  Shota Rustaveli  National Science Foundation
(Ref.: FR-18-10849) is gratefully acknowledged.

\vskip+2mm

\textit{Authors address: Andrea Razmadze Mathematical Institute of Tbilisi State University}, 
\textit{6 Tamarashvili Str., Tbilisi 0177, Georgia, e-mail: dali.zangurashvili@tsu.ge}

\end{document}